\documentclass[11pt]{article}
\usepackage[margin=1.1in]{geometry}
\usepackage{enumerate}
\usepackage{amssymb}
\usepackage{amsmath}
\usepackage{mathrsfs}
\usepackage{amsthm}
\usepackage{mathtools}
\usepackage{float}
\usepackage{braket}
\usepackage{cancel}
\usepackage{eufrak}

\newcommand{\R}{\mathbb{R}}

\usepackage{cleveref}
\newtheorem*{theorem*}{Theorem}
\newtheorem{theorem}{Theorem}[section]

\newtheorem{proposition}[theorem]{Proposition}

\newtheorem{lemma}[theorem]{Lemma}
\newtheorem{definition}[theorem]{Definition}
\newtheorem{open}[theorem]{Open Problem}
\newtheorem{conj}[theorem]{Conjecture}

\usepackage[backend=biber, style=alphabetic,isbn=false,doi=false]{biblatex}
\renewbibmacro{in:}{%
	\ifentrytype{article}{}{\printtext{\bibstring{in}\intitlepunct}}}
\DeclareLabelalphaTemplate{%
	\labelelement{%
		\field{citekey}
	}
}

\newcommand{\genlegendre}[4]{%
	\genfrac{(}{)}{}{#1}{#3}{#4}%
	\if\relax\detokenize{#2}\relax\else_{\!#2}\fi
}
\newcommand{\legendre}[3][]{\genlegendre{}{#1}{#2}{#3}}
\usepackage{tikz}
\usetikzlibrary{calc}

\definecolor{red1}{RGB}{230,25,75}
\definecolor{blue1}{RGB}{0,130,200}
\definecolor{yellow1}{RGB}{255, 225, 25}
\definecolor{green1}{RGB}{60,180,75}
\definecolor{darkgrey1}{RGB}{57, 59, 58}
\definecolor{vtx1}{RGB}{115, 110, 111}

\bibliography{PaleySystems}

\begin{document}

	\title{Irreducible Non-Metrizable Path Systems in Graphs}
	\author{Daniel Cizma\thanks{\raggedright Department of Mathematics, Hebrew University, Jerusalem 91904, Israel. e-mail: daniel.cizma@mail.huji.ac.il.} \and  {Nati Linial\thanks{School of Computer Science and Engineering, Hebrew University, Jerusalem 91904,
			Israel. e-mail: nati@cs.huji.ac.il~.	Supported by BSF US-Israel Grant 2018313 "Between Topology and Combinatorics". }}}

	\maketitle

\begin{abstract}
A path system $\mathcal{P}$ in a graph $G=(V,E)$ is said to be irreducible if there does not exist a partition $V= A\sqcup B$ such that $\mathcal{P}$ restricts to a path system on both $G[A]$ and $G[B]$. In this paper, we construct an infinite family of non-metrizable irreducible path systems defined on certain Paley graphs.
\end{abstract}
\section{Introduction}
A path system $\mathcal{P}$ in a graph $G= (V,E)$ is a collection of paths in $G$ such that for every $u,v \in V$ there is exactly one path in $\mathcal{P}$ connecting $u$ and $v$. We say that $\mathcal{P}$ is {\em consistent} if for every path $P$ in $\mathcal{P}$, any subpath of $P$ is also a path in $\mathcal{P}$. Every positive weight function $w: E\to \mathbb{R}^+$ gives rise to a consistent path system, by putting in $\mathcal{P}$ only $w$-shortest paths. A path systems that comes from such $w$ is said to be {\em metrizable}. We say that $G=(V,E)$ is metrizable if every consistent path system on $G$ is metrizable. The main findings so far in this newly emerging research area, can be be briefly described as follows: Metrizable graphs are very rare, yet all outerplanar graphs are metrizable. We encourage the reader to consult our paper \cite{CL} for a full account of what is currently known in this area, but note that the present paper is entirely self-contained. Some of these definitions and questions make sense also for {\em partial} path systems. Metrizability of partial path systems was investigated in \cite{Bo}.

Here we introduce the notion of an {\em irreducible} path system.
Such path systems are "atomic", in that they cannot be decomposed into two smaller consistent path systems.

\begin{definition}\label{def1}
Let $\mathcal{P}$ be a path system in a graph $G=(V,E)$. A partition $V = A\sqcup B$ with $A,B\neq\emptyset$ is called a {\em reduction} of $\mathcal{P}$ if all vertices of the path $P_{u,v}$ belong to $A$ (resp.\ $B$) whenever $u,v \in A$ (resp.\ $u, v\in B$). A path system with no reductions is said to be {\em irreducible}.
\end{definition}

We illustrate the notion of irreducibility using the Petersen graph $\Pi$ which is nonmetrizable. The diameter of $\Pi$ is $2$ and its girth is $5$. Therefore, every two vertices in $\Pi$ are connected by a single path of length $1$ or $2$. In $\mathcal{P}$, the path system that we construct, most pairs are connected by these native shortest paths. There are five exceptional pairs of nonadjacent of vertices which are connected in $\mathcal{P}$ by one of the five colored paths in \Cref{fig:nonMetPetersen}. It is easily verified that $\mathcal{P}$ is consistent, and as shown in \cite{CL}, it is non-metrizable. However, this path system admits a reduction $A,B$ where $A$ and $B$ are the set of dark and light vertices, respectively, as seen in \Cref{fig:nonMetPetersen}.
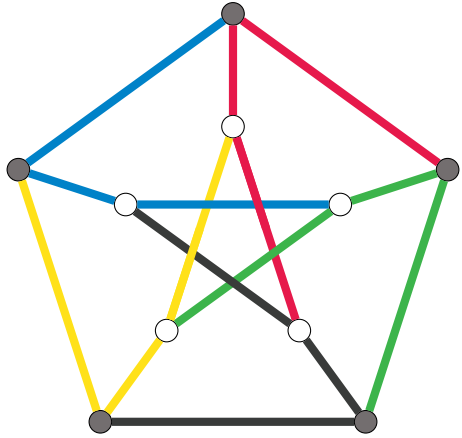
\begin{figure}[h]
\centering

	\begin{tikzpicture}[scale=0.6, every node/.style={scale=0.6}]

	\node[draw,circle,minimum size=.5cm,inner sep=1pt,fill = vtx1] (1) at (-0*360/5 +90: 5cm) [scale=1] {};
	\node[draw,circle,minimum size=.5cm,inner sep=1pt,fill = vtx1] (2) at (-1*360/5 +90: 5cm) [scale=1]{};
	\node[draw,circle,minimum size=.5cm,inner sep=1pt,fill = vtx1] (3) at (-2*360/5 +90: 5cm)[scale=1] {};
	\node[draw,circle,minimum size=.5cm,inner sep=1pt,fill = vtx1] (4) at (-3*360/5 +90: 5cm) [scale=1]{};
	\node[draw,circle,minimum size=.5cm,inner sep=1pt,fill = vtx1] (5) at (-4*360/5 +90: 5cm)[scale=1] {};

	\node[draw,circle,minimum size=.5cm,inner sep=1pt] (6) at (-5*360/5 +90: 2.5cm) [scale=1]{};
	\node[draw,circle,minimum size=.5cm,inner sep=1pt] (7) at (-6*360/5 +90: 2.5cm) [scale=1]{};
	\node[draw,circle,minimum size=.5cm,inner sep=1pt] (8) at (-7*360/5 +90: 2.5cm) [scale=1]{};
	\node[draw,circle,minimum size=.5cm,inner sep=1pt] (9) at (-8*360/5 +90: 2.5cm) [scale=1]{};
	\node[draw,circle,minimum size=.5cm,inner sep=1pt] (10) at (-9*360/5 +90: 2.5cm)[scale=1] {};

	\coordinate (c1) at ($(7)!0.5!(10)$);
	\coordinate (c2) at ($(6)!0.5!(8)$);
	\coordinate (c3) at ($(7)!0.5!(9)$);
	\coordinate (c4) at ($(8)!0.5!(10)$);
	\coordinate (c5) at ($(6)!0.5!(9)$);

	\draw [line width=3.1pt,-,red1] (1) -- (2);
	\draw [line width=3.1pt,-,blue1] (1) -- (5);
	\draw [line width=3.1pt,-,red1] (1) -- (6);
	\draw [line width=3.1pt,-,green1] (2) -- (3);
	\draw [line width=3.1pt,-,green1] (2) -- (7);
	\draw [line width=3.1pt,-,darkgrey1] (3) -- (4);
	\draw [line width=3.1pt,-,darkgrey1] (3) -- (8);
	\draw [line width=3.1pt,-,yellow1] (4) -- (5);
	\draw [line width=3.1pt,-,yellow1] (4) -- (9);
	\draw [line width=3.1pt,-,blue1] (5) -- (10);
	\draw [line width=3.1pt,-,red1] (6) -- (8);
	\draw [line width=3.1pt,-,yellow1] (6) -- (9);
	\draw [line width=3.1pt,-,green1] (7) -- (9);
	\draw [line width=3.1pt,-,blue1] (7) -- (10);
	\draw [line width=3.1pt,-,darkgrey1] (8) -- (10);

	\draw [line width=3.1pt,-,red1] (c2) -- (6);
	\draw [line width=3.1pt,-,yellow1] (c5) -- (9);

	\end{tikzpicture}
	\caption{A reducible non-metrizable path system in the Petersen Graph}
	\label{fig:nonMetPetersen}

	\end{figure}

The main result of the present paper is a construction of an infinite family of non-metrizable irreducible path systems. These systems are defined on certain Paley graphs $G_p$, where $p \equiv 1\pmod 4$ is a prime integer. The vertex set of $G_p$ is the finite field $\mathbb{F}_p$, and $a,b\in \mathbb{F}_p$ are neighbors iff $a-b$ is a quadratic residue in $\mathbb{F}_p$. We denote by $R$ and $N$ the sets of quadratic residues, resp.\ non-residues in $\mathbb{F}_p$.
\\
In order for our construction to work, we need to assume that $2, 3\in N$ are both quadratic non-residues in $\mathbb{F}_p$. Using quadratic reciprocity it is easy to see that if $p\equiv 5 \pmod{ 24}$, then $-1\in R$ and $2,3 \in N$. By Dirichlet's theorem on primes in arithmetic progressions, there are exists infinitely many primes $p\equiv 5 \pmod{ 24}$.

We define the following path system $\mathcal{P}_p$ on $G_p$: Let $a,b \in \mathbb{F}_p$

\begin{itemize}
	\item If $b-a\in R$ is a quadratic residue then $P_{a,b} \coloneqq (a,b)$
	\item If $b-a = 3$ then $P_{a,b} \coloneqq (a,a+1, a+2,b)$
	\item If $b-a\in N$ is a quadratic non-residue and $b-a\neq \pm 3$ then $P_{a,b} \coloneqq (a,\frac{b+a}{2},b)$
\end{itemize}
We prove:
\begin{theorem}\label{thm:main}
The path system  $\mathcal{P}_p$ is irreducible and non-metrizable for all primes $p>5$ for which $-1\in R$ is a quadratic residue and $2,3\in N$ are quadratic non-residues.
\end{theorem}
In fact, the proof that we present here shows that $\mathcal{P}_p$ is non-metrizable for $p\geq 1711$. That the same holds also when $5<p <1711$ can be directly verified by a computer.
\section{Proof of \Cref{thm:main}}
Let us start with the basic relevant concepts. A {\em path system} $\mathcal P$ in $G=(V,E)$ is a collection of simple paths in $G$ such that for every $u,v\in V$  there is exactly one member $P_{u,v}\in\mathcal P$ that connects between $u$ and $v$. We say that $\mathcal P$ is {\em consistent} if for every $P\in\mathcal{P}$ and two vertices $x,y$ in $P$, the $xy$ subpath of $P$ coincides with $P_{x,y}$.
Unless otherwise stated, every path system that we encounter here is consistent. A path system $\mathcal P$ is said to be metrizable if there exists some $w:E(G) \to (0,\infty)$ such that for each $u,v \in V$,  $w(P_{u,v}) \leq w(Q)$ for every $uv$ path $Q$.

We first prove the path system $\mathcal{P}_p$ is irreducible.
Notice that $\mathcal{P}_p$ is cyclically symmetric. Namely, if $(a_1,\dots, a_r)$ is a path in $\mathcal{P}_p$ then so is $(a_1 + x,\dots, a_r+x)$, for any $x\in \mathbb{F}_p$.

\begin{proposition} \label{prop:irreducible}
For $p\geq 30$, the path system $\mathcal{P}_p$ is irreducible.
\end{proposition}
\begin{proof}
Suppose toward contradiction that $V = A\sqcup B$ is a reduction of $\mathcal{P}_p$. Consider the Hamiltonian cycle $C = (0,1,2,3,\dots, p-1,0 )$ in $G_p$. This reduction splits this cycle into $2k$ segments $C = \bigcup_1^{2k} R_i$,
where $V(R_{2i-1}) \subset A$ and $V(R_{2i}) \subset B$ for all $1\leq i \leq k$. Let $l$ be the smallest length of these segments, where the length of a segment is the number of its edges. Due to $\mathcal{P}_p$'s rotational symmetry, we can and will assume that the shortest segment is $R_1$, and $R_{1} = ( -  \frac{l}{2} , \dots, -1, 0,1,\dots, \frac{l}{2}  )$ for even $l$ and $R_{1} = ( - \lfloor \frac{l}{2} \rfloor, - \lfloor \frac{l}{2} \rfloor + 1,\dots, -1, 0,1,\dots, \lfloor \frac{l}{2} \rfloor, \lfloor \frac{l}{2} \rfloor + 1 )$ for odd $l$.

If $l=1$, then $R_1 = (0, 1)$ and $-1\in R_{2k}$~, $2 \in R_{2}$. By assumption $V(P_{-1,2})\subset B$, since both $V(R_{2k}), V(R_2)\subset B$. However, $P_{-1,2} = (-1,0,1,2)$, whereas $V(R_1) = \{0, 1\}\subset A$. The same argument works as well when $l=0$.

We next consider the case where $l$ is even and $2\leq l \leq \frac{p}{4}$. Again, by symmetry we may assume that $R_{1} = ( - \frac{l}{2} ,   \dots, -1, 0,1,\dots, \frac{l}{2} )$. Since $l$ is minimal, $|R_2|\geq l$ and $|R_{2k}|\geq l$, and hence $(\frac{l}{2} + 1, \frac{l}{2} + 2\dots,\frac{3l}{2} + 1 ) \subseteq R_{2}$ and $(-\frac{l}{2}- 1,- \frac{l}{2}- 2\dots,-\frac{3l}{2} - 1 ) \subseteq R_{2k}$. We claim next that the range $[\frac{l}{2}+1, \frac{3l}{2} + 1]$ is comprised only of non-residues. For if $b\in [\frac{l}{2}+1, \frac{3l}{2} + 1]$, is a quadratic residue, then by construction $P_{-b,b} = (-b,0,b)$. This is a contradiction, since $-b\in V(R_{2k})\subset B, b\in V(R_1)\subset B$, $0\in V(R_0)\subset A$. On the other hand, the interval $[\frac{l}{2}+1, \frac{3l}{2} + 1]$ contains both $\frac{l}{2}+1$ and $l+2$, where one is a quadratic residue and the other is not, since their ratio is $2:1$, and by assumption $2\in N$. When $l$ is odd an identical argument works for the same range.

There remains the range $l >\frac{p}{4}$. Again, we assume that $l$ is even since the proof when $l$ is odd is essentially identical. As $l$ is the length of the smallest segment and $C$ contains an even number of segments, it necessarily follows that $C$ splits into
exactly two segments $C=R_1\sqcup R_2$, where  $R_{1} = ( - \frac{l}{2}, \dots, -1, 0,1,\dots, \frac{l}{2} )$, $l \leq \frac{p}{2}$.

If $P_{-a,a+1}=(-a,\frac{p+1}{2},a+1)$ for some small $a>0$, we again encounter a contradiction, since this path starts and ends in $V(R_1)=A$ and its middle vertex is in $V(R_2)=B$. The choice $a=1$ won't do, since $P_{-1,2}=(-1,0,1,2)$. Also, for $P_{-a,a+1}=(-a,\frac{p+1}{2},a+1)$ to hold, $2a+1$ must be a non-square, for otherwise $P_{-a,a+1}=(-a,a+1)$. Thus, if $5\in N$, we can use $a=2$. If not, and $5\in R$, then $15\in N$ because we are assuming $3\in N$. We can, therefore, take $a=7$, which is "small enough" under the assumption $p\geq 30$.

Let $L_p$ be the maximum length of a consecutive segment of non-residues in $\mathbb{F}_p$. We note that
upper bounds on $L_P$ can be used to derive shorter proofs of \Cref{prop:irreducible}. E.g., Hummel \cite{Hu} showed that $L_p\le\sqrt p$ for every prime $p\neq 13$, and Burgess \cite{B2} proved that $L_p\le O(p^{1/4}\log p)$.
\end{proof}

It remains to show that:
\begin{proposition}\label{prop:PaleyNonMet}
The path system  $\mathcal{P}_p$ is not metrizable.
\end{proposition}

Let $\varphi_x$ be the rotation-by-$x$ map
$$\varphi_x(a_1,\dots, a_r) = (a_1 + x, \dots, a_r + x).$$
Again we use the cyclic symmetry, i.e., invariance under $\varphi_x$ of $G_p$ and $\mathcal{P}_p$.

Had $\mathcal{P}$ been metrizable, there would exist a weight function $w:E\to (0,\infty)$ such that
\begin{equation}\label{metrizability}
w(P_{u,v} ) \leq w(Q)\text{~for every~}u,v \in V \text{~and every~} uv \text{~path~} Q.
\end{equation}

Due to cyclic symmetry, if $w$ satisfies (\ref{metrizability}), so does $w\circ \varphi_x$. Moreover, the set of all $w$ that satisfy (\ref{metrizability}) is a convex cone. Therefore, if $w$ satisfies (\ref{metrizability}), then so does the weight function
$$\tilde{w} = \sum_{x\in \mathbb{F}_p} w\circ \varphi_x.$$
Note also that $\tilde{w}(x,y)$ depends only on $|x-y|$.  \\

We associate a formal variable $x_a$ with every quadratic residue $a\in R$. By the above discussion, if $\mathcal{P}_p$ is metrizable then the following system of linear equations and inequalities is feasible:

\begin{alignat}{3}\label{eq:LP1}
	2x_{a}& \leq x_{b} + x_{c}\qquad&& a,b,c\in R,~~ 2a = b  + c \neq 3 \nonumber \\
	3x_{1}& \leq x_{b} + x_{c}\qquad&& a,b,c\in R,~~ 3 = b  + c \tag{$\ast$} \\
	x_{a} & = x_{-a}\qquad&& a\in R \nonumber\\
	x_a & > 0    \qquad && a \in R  \nonumber
\end{alignat}
Therefore, to show that $\mathcal{P}_p$ is non-metrizable it suffices to show that (\ref{eq:LP1}) is infeasible.\\

We consider instead a more general setup and ask if the following system of linear inequalities is feasible. It has $M+N$ inequalities in $x_1,\dots, x_N $, and it is given that $a_1,\ldots,a_M\ge 2$.
\begin{alignat}{3}\label{eq:LP2}
	a_{m} x_{i_m} & \leq x_{j_m} + x_{k_m} \qquad&& m = 1,\dots, M  \tag{$\ast\ast$} \\
	x_i & > 0    \qquad &&i =1,2,\dots, N. \nonumber
\end{alignat}
We associate a digraph $\Vec{D}$ with this system of linear inequalities. Its vertex set is $V(\Vec{D})=1,2,\dots, N$. The edge set is defined as follows: Any inequality $a_{m} x_{\alpha}  \leq x_{\beta} + x_{\gamma}$ that appears in the system gives rise to an edge from $\alpha$ to $\beta$ and one from $\alpha$ to $\gamma$. We say the system of inequalities is {\em strongly connected} if $\Vec{D}$ is strongly connected. We observe:
\begin{lemma} \label{lem:strongDirLP}
If \eqref{eq:LP2} is a strongly connected system, then it is feasible if and only if $a_1 = a_2 = \cdots = a_M = 2$. In this case $x_1 = x_2 = \cdots = x_N$ is the only feasible solution.
\end{lemma}
\begin{proof}
Clearly, $x_1 = x_2 = \cdots = x_N$ is a feasible solution when $a_1 = a_2 = \cdots = a_M = 2$.\\
Consider a feasible solution $(x_1,\dots, x_N) \in \R^N$ and say that $x_\alpha=\max_{1\leq j \leq N} x_j$. If there is a directed edge $(\alpha,\beta) \in E(\Vec{D})$, then some inequality in the system must have the form $a x_{\alpha} \leq x_{\beta} + x_{\gamma}$. Here $a\ge 2$, and since $x_{\alpha}$ is maximal, necessarily $a = 2$ and $x_{\alpha} = x_{\beta} = x_{\gamma}$. By the same reasoning $x_{\alpha} = x_{\delta}$ whenever there is a directed path from $\alpha$ to $\delta$ in $H$. The claim follows since $\Vec{D}$ is strongly connected.
\end{proof}
We show that the following subsystem of \eqref{eq:LP1} is strongly connected, and hence
by \cref{lem:strongDirLP}, it is infeasible. Let $R' \coloneqq R\setminus \Big\{\frac{p+3}{2}, \frac{p-3}{2}\Big\}$
\begin{alignat}{4}\label{eq:LP3}
I_{a; b,c}:\qquad&&	2x_{a}& \leq x_{b} + x_{c}\qquad&& a,b,c\in R', ~~~2a = b  + c \neq 3 \nonumber \\
I_1:\qquad&&	3x_{1}& \leq x_{b} + x_{c}\qquad&& a,b,c\in  R' , ~~~3 = b  + c \tag{$\ast\ast\ast$} \\
I_0:\qquad&&		x_a & > 0    \qquad && a \in R'  \nonumber
\end{alignat}
The following theorem of Burgess \cite{B1,B2} shows that Paley graphs resemble random $G(n, 1/2)$ graphs
(see also, e.g., \cite{BEH,AC}). This theorem is stated in terms of the Legendre's Symbol. Namely, for $0\neq a\in \mathbb{F}_p$ we write $\legendre{a}{p}=\pm 1$ to indicate whether or not $a$ is a quadratic residue.

\begin{theorem}\label{thm:BurgessBound}
 	Let $a_1,\dots, a_k \in \mathbb{F}_p$ be distinct. Then
 	$$\bigg|\sum_{x=0}^{p-1} \prod_{i=1}^{k}\legendre{x-a_i}{p} \bigg| \leq (k-1)\sqrt{p}$$
\end{theorem}
We denote neighbor sets in $G_p$ by $\Gamma$.
\begin{lemma}\label{lem: commonNeighbors}
	Let $x,y,z \in G_p$ distinct vertices. Then
$$\big|(\Gamma(x) \cap \Gamma(y)) \setminus \Gamma(z)\big| = \frac{p}{8} + O(\sqrt{p})$$
in fact
$$\left|\big|(\Gamma(x) \cap \Gamma(y)) \setminus \Gamma(z)\big| - \frac{p}{8}\right| \le 5 \sqrt{p} + 1$$
\end{lemma}
\begin{proof}
Recall that $z\in \Gamma(a)$ iff $a-z\in R$. Therefore
\begin{equation*}
	\begin{split}
	\big|(\Gamma(x) \cap \Gamma(y)) \setminus \Gamma(z)\big| &= \frac{1}{8}\sum_{\substack{i=0\\
				i\neq x,y,z}}^{p-1} \bigg( 1+\legendre{x-i}{p}\bigg)\bigg( 1+\legendre{y-i}{p}\bigg)\bigg( 1-\legendre{z-i}{p}\bigg)  \\
			&\qquad \qquad \qquad + \frac{1}{4}\bigg( 1+\legendre{x-z}{p}\bigg)\bigg( 1+\legendre{y-z}{p}\bigg)\\
			 &\leq  \frac{1}{8}\sum_{i=0}^{p-1} \bigg( 1+\legendre{x-i}{p}\bigg)\bigg( 1+\legendre{y-i}{p}\bigg)\bigg( 1-\legendre{z-i}{p}\bigg)  + 1
	\end{split}
\end{equation*}
Expanding the sum we get
$$\frac{p}{8} + \sum_{i=0}^{p-1} \left[ \legendre{x-i}{p} \legendre{y-i}{p}- \legendre{x-i}{p} \legendre{z-i}{p}  - \legendre{y-i}{p} \legendre{z-i}{p} - \legendre{x-i}{p} \legendre{y-i}{p}\legendre{z-i}{p} \right].$$
The conclusion follows from \cref{thm:BurgessBound}.
\end{proof}
\begin{lemma} \label{lem:strongConn}
The system  \eqref{eq:LP3}  is strongly connected for large $p$.
\end{lemma}
\begin{proof}
To prove this lemma we show that the following digraph $\Vec{D}$ on the vertex set $R'$ is strongly connected. Every triple $x,y,z\in R'$ with $2x = y+z\neq 3$ gives rise to the edges $(x,y), (x,z)\in E(\Vec{D})$. Also, whenever $y+z=3$ for $y\neq z$ in $R'$ we get the edges $(1,y), (1,z)\in E(\Vec{D})$.

For any two vertices $a,b\in R'$, we find a path in $\Vec{D}$ from $a$ to $b$. We first assume $4a\neq b$ and we
find three distinct elements $\alpha, \beta, \gamma \in R'$
such that $2a = \alpha + \beta$ and $2\beta = \gamma + b$. By construction, $(a,\beta),(\beta,b)\in E(\Vec{D})$ is a $2$-step path from $a$ to $b$. Let $S$ be the set of $\beta \in R$ for which there exist $\alpha, \gamma \in R$ with $2a = \alpha + \beta$ and $2\beta = \gamma + b$.
$$S = (\Gamma(2a) \cap R) \setminus \Gamma(b/2)$$
Indeed, fix $\beta \in R$. Notice that $2a = \alpha + \beta$ for some $\alpha \in R$ iff $\beta \in \Gamma(2a) $. Similarly $2\beta = \gamma + b$ for some $\gamma \in R$ iff $2\beta \in \Gamma(b) $, and since $2$ is a non-residue $2\beta \in \Gamma(b) $  iff $\beta \notin \Gamma(b/2) $. As $R = \Gamma(0)$ and $0,2a,\frac{b}{2}$ are all distinct we can use \cref{lem: commonNeighbors} and to get
$$ |S| = |(\Gamma(2a) \cap R) \setminus \Gamma(b/2)| \ge \frac{p}{8} - 5\sqrt{p} - 1.$$
But aside of the requirement that $\beta\in S$ we also insist that $\beta, 2a - \beta, 2\beta - b \notin \Big\{\frac{p+3}{2}, \frac{p-3}{2}\Big\}$. This rules out at most $6$ eligible values for $\beta$. Such a $\beta$ exists provided that $\frac{p}{8} - 5\sqrt{p} -1 > 6$. This inequality holds when $p\ge 1711$. \\
When $4a  = b$ there is a two step path from $a$ to $b$ via $\beta$ for any $\beta \in R' \setminus \{a,b\}$.
\end{proof}
It follows from \Cref{lem:strongConn} that for large enough $p$,  $p\geq 1711$, $\mathcal{P}_p$ is non-metrizable. We remark that with a bit more work the bound of $5\sqrt{p} + 1$ in \Cref{lem: commonNeighbors} can be improved to $2\sqrt{p} + 2$. This implies, in turn, that $\mathcal{P}_p$ is non-metrizable for $p\geq 379$.
\section{Open Problems}
The notion of an irreducible path systems suggests many open questions and conjectures. Here are some of them:
\begin{conj}
Asymptotically almost every graph has an irreducible non-metrizable path system. Specifically, this holds with probability $1-o_n(1)$ for $G(n,1/2)$ graphs.
\end{conj}

Perhaps even more is true. Let $\Pi_G$ be the collection of all consistent path systems in $G$, and let $\mathcal{R}_G\subseteq \Pi_G$ be the collection of all those which are metrizable or reducible.
\begin{open}
Is it true that $|\mathcal{R}_G| = o_n(|\Pi_G|)$ for asymptotically almost every $n$-vertex graph?
\end{open}
If we wish to know whether a given path system is irreducible, we currently must resort to brute force searching. So we ask:
\begin{open}
Is there an efficient algorithm to determine if a given path system is irreducible?
\end{open}

A reduction $V = A\sqcup B$ as in Definition \ref{def1} still does not tell the whole story.

\begin{open}
Let $G=(V,E)$ be a graph, a bipartition $V=V_1\sqcup V_2$ of its vertex set, and path systems
$\mathcal{Q}_1, \mathcal{Q}_2$ on $G(V_1), G(V_2)$ respectively. Under what conditions is there is a consistent non-metrizable path system on $G$ whose restriction to $G_1, G_2$ coincides with $\mathcal{Q}_1, \mathcal{Q}_2$ respectively?
\end{open}

\newpage

\printbibliography
\end{document}